\documentclass[reqno]{amsart}%

\usepackage[utf8]{inputenc}
\usepackage[T1]{fontenc}
\usepackage[english]{babel}
\usepackage{amsmath}
\usepackage{amssymb}
\usepackage{amsthm}
\usepackage{systeme}
\usepackage{faktor} 
\usepackage{mathtools}

\numberwithin{equation}{section}
\usepackage[pagewise]{lineno}
\usepackage{amsaddr}%

\usepackage{enumerate} 
\usepackage{graphicx}  
\usepackage[shortlabels]{enumitem}
\usepackage{bm} %
\usepackage{emptypage} %
\usepackage{hyperref}
\usepackage{tikz-cd}

\usepackage{csquotes}

\newtheorem{theorem}{Theorem}[section]
\newtheorem{corollary}[theorem]{Corollary}

\newtheorem{lemma}[theorem]{Lemma}
\newtheorem{proposition}[theorem]{Proposition}

\newtheorem{taggedtheoremx}{Theorem}
\newenvironment{taggedtheorem}[1]
 {\renewcommand\thetaggedtheoremx{#1}\taggedtheoremx}
 {\endtaggedtheoremx}

\theoremstyle{definition}
\newtheorem{definition}[theorem]{Definition}
\newtheorem{remark}[theorem]{Remark}
\newtheorem*{notation}{Notation}

\newtheorem{ex}[theorem]{Example} 

\DeclareMathOperator{\Ima}{Im}

\DeclareMathOperator{\Ide}{Id}

\usepackage{fancyhdr} 
\usepackage{todonotes}

\fancypagestyle{plain}{
\fancyhead[L]{}
\fancyhead[C]{}
\fancyhead[R]{}
\fancyfoot[L]{}
\fancyfoot[C]{\thepage}
\fancyfoot[R]{}

}

\usepackage{xparse}
\makeatletter
\RenewDocumentCommand{\title}{om}{%
   \IfNoValueTF{#1}
     {\gdef\shorttitle{GENERALISATION OF BABBAGE EQUATION}}
     {\gdef\shorttitle{#1}}%
   \gdef\@title{#2}%
}
\makeatother

\usepackage[firstpage]{draftwatermark}
\SetWatermarkText{Preprint of: doi:10.3934/dcds.2020303 }
\SetWatermarkFontSize{0.85cm}
\SetWatermarkScale{1}
\SetWatermarkAngle{0}
\SetWatermarkColor[gray]{0.4}
\SetWatermarkVerCenter{2.5cm}

\begin{document}

\title{A generalisation of the Babbage functional equation}

\author{Marc Homs-Dones}
\address{Mathematical Institute, University of Oxford, England.}


\begin{abstract}
    A recent refinement of Ker\'ekj\'art\'o's Theorem has shown that in $\mathbb R$ and $\mathbb R^2$ all $\mathcal C^l$--solutions of the functional equation $f^n =\textrm{Id}$ are $\mathcal C^l$--linearizable, where $l\in \{0,1,\dots \infty\}$. When $l\geq 1$, in the real line we prove that the same result holds for solutions of $f^n=f$, while we can only get a local version of it in the plane. Through examples, we show that these results are no longer true when $l=0$ or when considering the functional equation $f^n=f^k$ with $n>k\geq 2$.
\end{abstract}



\date{\today}

\maketitle

\section{Introduction}

In 1815, Babbage proposed the systematical study of $n$th order functional equations, i.e.
\begin{equation}
    F(x,f(x),f^2(x),\dots, f^n(x))=0,
    \label{Babbage_eq_general}
\end{equation}
where solutions $f$ are searched for a  given  $F$ and $f^n=f\circ f^{n-1}$ with $f^0=\Ide$. Already in his paper \cite{Babbage_original},
he emphasized the particular case,
\begin{equation}\label{babbage_eq}
    f^n= \Ide,
\end{equation}
which is known as \emph{Babbage's functional equation} and has been intensively investigated up until now
(see \cite{Baron2001,kerejarto_smooth}).
The solutions of \eqref{babbage_eq} are called \emph{periodic functions} or $n$th iterative roots of the identity, and their behavior depends greatly on the regularity of $f$ and its definition set.

We will mostly limit our study to functions $f\in \mathcal C^l$ defined on manifolds.
Moreover, we will only worry about the dynamics
defined by $f$, and thus we will study the functions up to conjugacy. In this area, there are many classical results which state that in $\mathbb R$, $\mathbb R^2$, $\mathbb S^1$ and $\mathbb S^2$ all solutions of \eqref{babbage_eq} are \emph{linearizable}, i.e. are topologically conjugated to a linear map (see Section \ref{review_periodic}).

The main goal of this paper is to find a similar classification
for the functional equation,
\begin{equation}
    f^n= f^k,
    \tag{$\star$}
    \label{ f^n= f^k}
\end{equation}
  where $n,k\in \mathbb N\cup \{0\}$ and $n>k$, which clearly is a generalization of the Babbage's functional equation, but in fact, it is a particular case of \eqref{Babbage_eq_general}.  Notice that when $f$ is bijective, equations \eqref{ f^n= f^k} and \eqref{babbage_eq} are equivalent.
Solutions of \eqref{ f^n= f^k} appear frequently in many sciences, especially solutions of $f^2=f$. Some simple examples are projections and rounding functions. A more complex one is the Hardy-Weinberg equilibrium, which is frequently used in biology (see  \cite[Chapter 11]{Population_dynamics}).
In a two allele population this equilibrium states that the function
\[f(p,q)=\left (\left(p+\frac{q}{2}\right )^2,2\left(p+\frac{q}{2}\right )\left (1-p-\frac{q}{2}\right )\right ),\]
is \emph{idempotent}, i.e. $f^2=f$. Moreover, in \cite{Yap_simpler} it is shown that a more refined version of this equilibrium gives rise
to a map such that $f^3=f^2$ and $f^2\not =f$.

\subsection{Main results}

We will always assume that $n>1$, except in Section \ref{review_periodic}. A function $f$ is in $\mathcal{C}^l$ if it is  $l$--times continuously differentiable. If $l=\infty$ we say that $f$ is smooth and if $f$ is analytic we say it is in $\mathcal C^\omega$. Note that differentiable functions do not need to be in $\mathcal C^1$.

\begin{definition}
A function $f:U\subset \mathbb R^m\rightarrow U$ is $\mathcal C^l$--\emph{linearizable} if it is $\mathcal C^l$--conjugated to a linear map $L:\mathbb R^m\rightarrow \mathbb R^m$. That is, there exists a $\mathcal C^l$--diffeomorphism $\varphi:U \rightarrow \mathbb R^m$ which conjugates $f$ with $L$, i.e.  $ \varphi \circ f = L  \circ \varphi $.
\label{def_lin}

\end{definition}

In the one-dimensional case our main results are:

\begin{taggedtheorem}{A}
Let $f:\mathbb R \rightarrow  \mathbb R$ be a differentiable function, such that $f^n=f$, then  $f$ is differentiably linearizable. Moreover, if $f\in \mathcal C^l$, with $l\in \{1,\dots,\infty,\omega\}$, then it is $\mathcal C^l$--linearizable.
\label{theorem_A}
\end{taggedtheorem}

\begin{taggedtheorem}{B}
Let $f:\mathbb R\rightarrow \mathbb R$ be an analytic function satisfying \eqref{ f^n= f^k}. Then $f$ is $\mathcal C^\omega$--linearizable.
\label{theorem_B}
\end{taggedtheorem}

In Example \ref{idemp_exemple} we show examples of continuous functions such that  $f^n=f^k$ for all $n>k\geq 1$ and are not linearizable. Moreover, in Example \ref{idemp_exem_smooth} we show examples of smooth functions such that $f^n=f^k$ for all $n>k\geq 2$ and are not linearizable. Thus, the previous results are sharp. Furthermore, for both cases, we give an uncountable family of solutions not topologically conjugated to each other.

In the two-dimensional case we prove:

\begin{taggedtheorem}{C}
Let $f:\mathbb R ^2\rightarrow \mathbb R^2$ be smooth, non-periodic and non-constant. Then, $f^n=f$ if and only if  $f^3=f$. Moreover, up to smooth conjugacy we have:
\begin{itemize}
\item the solutions of $f^2=f$ are exactly $g(x,y)= (x+y\mathfrak g(x,y),0)$ with $\mathfrak g\in \mathcal C^\infty$ an arbitrary function;
\item the solutions of  $f^3=f$ and  $f^2\not = f$ are exactly $g(x,y)= (-x+y\mathfrak g(x,y),0)$ with $\mathfrak g\in \mathcal C^\infty$ an arbitrary function.
\end{itemize}
\label{rectifica}
\end{taggedtheorem}

\begin{taggedtheorem}{D}
Let $f:\mathbb R^2\rightarrow\mathbb R^2$ be a $\mathcal C^l$--solution of $f^n=f$ with $l\in \{1,\dots, \infty\}$. Then, in a neighborhood of $\Ima f$, $f$ is $\mathcal C^l$--linearizable.
\label{local_linearizable}
\end{taggedtheorem}

We also get partial results for the general case $f^n=f^k$ and we show in Example \ref{polinomis_no_proj} that there exist non globally  linearizable polynomial functions such that $f^n=f^k$ for all $n>k\geq 1$. Furthermore, the map defined in the forthcoming equation \eqref{e^1/x} shows that Theorem \ref{local_linearizable} does not hold for general solutions of \eqref{ f^n= f^k}.

In $\mathbb R^m$ we show that if $f$ is a $\mathcal C^l$--solution of \eqref{ f^n= f^k},  with $l\in \{1,\dots,\infty\}$  and $d=\dim(\Ima f^k)\leq 2$, then $\Ima f^k$ is a $\mathcal C^l$--submanifold diffeomorphic to $\mathbb R^d$ (see Proposition \ref{Imaf^k_R^m}).
Moreover, we show that in this case the only obstruction from getting generalizations of Theorem \ref{rectifica} and \ref{local_linearizable} is the fact that $\Ima f^k$ may be knotted in $\mathbb R^m$. In Section \ref{manifolds}, we explore basic properties of  solutions of \eqref{ f^n= f^k} when defined on manifolds (we use this term for manifolds without boundary).

Finally, in Section \ref{section_hardy} we show that not only the Hardy-Weinberg equilibrium in any dimension is an idempotent process, but it is also conjugated to a projection. We also show that the generalization given in \cite{Yap_simpler} is not linearizable.

\subsection{Review of the periodic case}
\label{review_periodic}

\begin{definition}
Given a periodic function $f$, its \emph{period} is the minimum $n>0$ for which $f^n=\Ide$.

Functions with period 2 are called  \emph{involutions}.

We will use  $I$ to denote any type of proper interval, i.e. with at least two distinct points.
\end{definition}

In the real line the following results have been known for a long time (see for instance \cite[Corollary 1]{fn_f} and \cite{f2_Id}).

\begin{proposition}
Let $n\in \mathbb N\setminus \{0\}$ and $f:I\rightarrow I$  be continuous. If $n$ is odd, $f^n=\Ide $ if and only if $f=\Ide$. If $n$ is even, $f^n=\Ide $ if and only if $f^2=\Ide$. Moreover, $f$ is decreasing if and only if it is an involution, i.e. $f^2=\Ide$ and $f\not =\Ide $.
\label{kerejarto_1dim_I}
\end{proposition}

\begin{remark}
If $f$ is a differentiable involution by the chain rule we have \break${f'(f(x))f'(x)=1}$ and since $f$ is decreasing, $f'<0$.
\label{f'<0}
\end{remark}

Now considering the conjugation given by $f-\Ide$ and the Inverse Function Theorem (for differentiable injective functions) it is easy to prove:

\begin{proposition}
Let $l\in \{0,\dots, \infty,\omega\}$, all $\mathcal C^l$--involutions defined in $\mathbb R$ are $\mathcal C^l$--conjugated to $-\Ide$. The same is true in  the class of differentiable functions.
\label{kerejarto_1dim_I_invol}
\end{proposition}
In the continuous case, the periodic functions on the circle are described in the following result, which is a straightforward consequence of the main theorem in~\cite{Babbage_circle}.

\begin{theorem}
Let $f:\mathbb S^1\rightarrow \mathbb S^1$ be a continuous function of  period $n$. Then, $f$ is a homeomorphism and,

\begin{itemize}
    \item if $f$ is order-preserving, $f$ is topologically conjugated to a rotation of angle $2\pi r/n$ where $r$ and $n$ are coprimes;
    \item if $f$ is order-reversing, $f$ is topologically conjugated to reflection through the $x$--axis.
\end{itemize}
\label{S^1 periodic}
\end{theorem}

As far as we know no classification has been found in the  class $\mathcal C^l$. In two dimensions we have Kerékjártó's Theorem.

\begin{theorem}
Let $f:\mathbb R^2\rightarrow \mathbb R^2$ be a $n$ periodic $\mathcal C^l$--function with $l\in \{0,\dots, \infty\}$.
Then, $f$ is $\mathcal C^l$--conjugated to a rotation of angle $2\pi r /n$ (with $r$ and $n$ coprimes) centered at the origin or a reflection through the $x$--axis.

\label{kerejarto_R2}
\end{theorem}

The $\mathcal C^l$ case is a recent result presented in \cite{kerejarto_c1,kerejarto_smooth}, whereas the continuous case is a classical result   \cite{Kerejarto_original}. A modern approach for the continuous case can be found in \cite{Kerejarto_constantin} where first an analogous result in the closed disc is seen.
Then, Theorem \ref{kerejarto_R2} follows from the study of periodic functions in the sphere.

\begin{theorem}
Let $f:\mathbb S^2\rightarrow \mathbb S^2$ be a continuous periodic function. Then, $f$ is topologically conjugated to an element of the orthogonal group $O(3)$.
\label{kerejarto en S^2}
\end{theorem}

Informally, all previous results can be restated as: in the manifolds $\mathbb R$, $\mathbb R^2$, $\mathbb S^1$ and $\mathbb S^2$, all periodic solutions are ``linearizable''. In this case, by ``linearizable'' in $\mathbb S^m$  we mean that they are conjugated to a linear map restricted to $\mathbb S^m\subset \mathbb R^{m+1}$ instead of the definition given in Definition \ref{def_lin}.

In higher dimensions such results do not exist.
For instance, in $\mathbb R^3$ there are periodic homeomorphisms such that their set of fixed points form a wild plane. It is well known that an homeomorphism of $\mathbb R^3$  cannot send a wild plane to an affine one, and thus these periodic functions cannot be linearizable. In \cite{Inequivalent_families} an uncountable number of such functions not conjugated to each other are presented. In fact, it is shown that for every period there are uncountable many equivalence classes up to topological conjugacy. Similar results are presented when considering periodic homeomorphisms in $\mathbb S^3$.

In the differential case we do not have a classification either. For instance, in $\mathbb R^7$ if $n>1$ is not a prime power there are uncountable many topological equivalent classes of period $n$ without any fixed point (see \cite{no_fix_R7}).
Since all linear maps have a fixed point, these cannot be linearizable.

\subsection{Elementary properties}

Given a set $A$ and $f:A  \rightarrow A$ it is easy to check that $\Ima f^{l+1} \subset\Ima f^l$ for all  $l\in \mathbb N$. Hence, $(f_{|\Ima f^l})^r$ is well defined and $(f_{|\Ima f^l})^r=(f^r)_{|\Ima f^l}$. We will denote this function by  $f^{r}_{|\Ima f^l}$. Moreover, we will write $f^{r}_{|\Ima f^l}=\Ide$ if $f^{r}_{|\Ima f^l}:\Ima f^l\rightarrow \Ima f^l $ is the identity.

We proceed to state several basic properties of any function that solves \eqref{ f^n= f^k}.

\begin{proposition}
Let $A$ be a set and $f:A  \rightarrow A$, then $f$ satisfies \eqref{ f^n= f^k} if and only if $f^{n-k}_{|\Ima f^k} =\Ide$. In this case,  $f_{|\Ima f^k}:\Ima f^k\rightarrow \Ima f^k$ is bijective.

\label{resultat_conjunts}
\end{proposition}
\begin{proof}
We have the following equivalences,
\begin{equation*}
    f^n=f^{n-k}\circ f^k = f^k
    \Longleftrightarrow \forall y\in \Ima f^k, \hspace{2mm} f^{n-k}(y)=y \Longleftrightarrow f^{n-k}_{|\Ima f^k} =\Ide.
\end{equation*}
Since the domain and codomain of $f$ coincide we have $\Ima f^{l+1}\subset \Ima f^l$ for all $l\in \mathbb N$. Thus, $f_{|\Ima f^k}:\Ima f^k\rightarrow \Ima f^k$ is well defined and $f^{n-k-1}_{|\Ima f^k}$ is its inverse.
\end{proof}

The above characterization will be key to study the solutions of \eqref{ f^n= f^k}. Indeed, if we consider $A$ as a set without any further structure, then all solutions of \eqref{ f^n= f^k} can be constructed in the following two steps. First, we choose any subset $B\subset A$ and any periodic function $f_1:B\rightarrow B$ with period dividing $n-k$. Secondly, we choose any function $f_2 : A\setminus B\rightarrow A$ such that $\Ima f_2^k\subset B$. Then, the function $f$ defined as $f_{|B}=f_1$ and $f_{|A\setminus B}=f_2$ satisfies \eqref{ f^n= f^k} and all solutions of the equation are of this form.

We now state two particular cases of Proposition \ref{resultat_conjunts} which will be especially useful.

\begin{remark}
If $f$ is surjective then $\Ima f^k=A$ and $f$ is a periodic function.
\label{exaustiva}
\end{remark}

\begin{remark}
If $f$ solves the equation \eqref{ f^n= f^k} and $f^k$ is constant, i.e. $\Ima f^k$ is a singleton, then $f_{|\Ima f^k}=\Ide$, and thus $f^{k+1}=f^k$.
\label{obs_|Ima f^k|=1}
\end{remark}

The following result is a direct consequence of the characterization given by Proposition \ref{resultat_conjunts}.

\begin{corollary}
Let $A$ be a set and $f:A \rightarrow A $ satisfying \eqref{ f^n= f^k}. Then, $f$ is also a solution of $f^{l_1(n-k)+k}=f^{l_2(n-k)+k}$ for any $l_1,l_2\in \mathbb N\cup \{0\}$.
\label{f^{l_1(n-k)+k}}
\end{corollary}

Hence, if $f$ is an idempotent function, i.e. $f^2=f$, it will be a solution of $f^n=f^k$ for all $n> k\geq 1$.

\begin{proposition}
 Let $A$ be a set and $f:A \rightarrow A $ satisfying \eqref{ f^n= f^k}. Then, the function $h=f^{k(n-k)}$ is idempotent, i.e. $h^2= h $, and $\Ima h =\Ima f^k$.
\label{h^2= h}
\end{proposition}

\begin{proof}
Using the previous corollary with $l_1=k$ and $l_2=0$ we get $ f^{k(n-k+1)}=f^k$. Now, applying $f^{k(n-k-1)}$ to both sides we have,
\[h^2=f^{2k(n-k)}=f^{k(n-k)}=h.\]
Finally, $\Ima h=\Ima f^k$ due to Proposition \ref{resultat_conjunts}.
\end{proof}

\begin{definition}
We say that  $r:A\rightarrow A$ is a \emph{retraction} if $r$ is continuous and idempotent, i.e. $r^2= r$. In this case we also say that $\Ima r$ is a \emph{retract} of $A$.
\end{definition}
We can restate Proposition \ref{h^2= h} in the continuous case as the following remark:
\begin{remark}
 If $f:A\rightarrow A$ is continuous and satisfies \eqref{ f^n= f^k}, then  $\Ima f^k$ is a retract of $A$. Moreover, if $A$ is Hausdorff by \cite[Page 233, Exercise 4]{Topology_Munkres}, $\Ima f^k$ is closed.
\label{obs_retract}
\end{remark}

\section{One-dimensional case}

We start discussing the simplest non-periodic case. Let $f:I \rightarrow I$ be a continuous idempotent function, i.e. $f^2=f$. Clearly $\Ima f$ is connected,  and by Remark \ref{obs_retract} it is closed in $I$.

Assume now for simplicity that $\Ima f=[a,b]$ (the general case can be seen in \cite{fn_f}). Then, the graphic of $f$ is bounded by the horizontal lines $\{y=a\}$ and $\{y=b\}$. Moreover, by Proposition \ref{resultat_conjunts}, $f_{|\Ima f}=\Ide$ and thus, $f$ has the general shape shown in Figure \ref{Figura_grafica_f}. Visually, it is clear that in general these functions cannot be differentiable. Formally, using the limit definition of the derivative at $a$ and $b$, the following result is straightforward.

\begin{figure}[htbp]
    \centering
        {\includegraphics[width=0.81\textwidth]{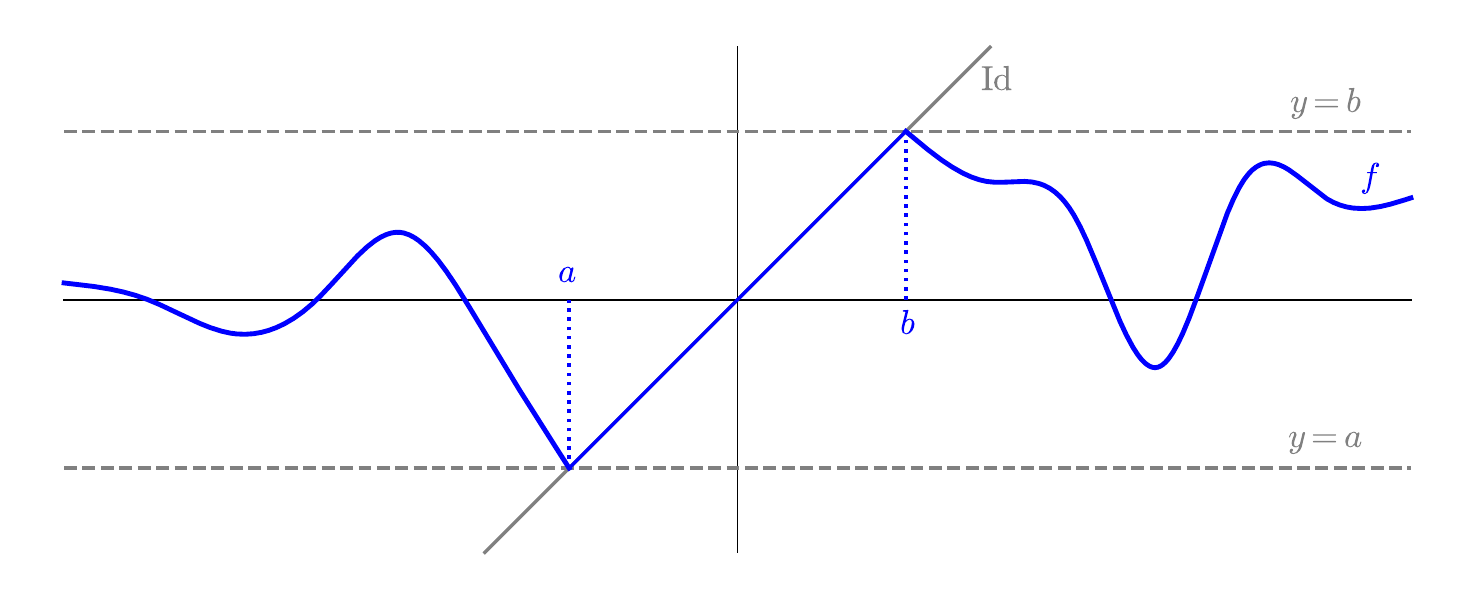}}
    \caption{Graph of a generic idempotent continuous function in $\mathbb R$.}
    \label{Figura_grafica_f}
\end{figure}

\begin{proposition}
Let $f:I \rightarrow  I$ be an idempotent differentiable function, then  $f$ is constant
or the identity.
\label{idem_dif}
\end{proposition}

For the functional equation $f^3=f$ all the previous arguments hold, except that $f_{|\Ima f}$ could also be an involution. Using Remark \ref{f'<0}, which states that involutions have strictly negative derivative, it is not difficult to prove:

\begin{proposition}
Let $f:I \rightarrow  I$ be a differentiable solution of $f^3=f$, then $f$ is constant,  the identity or an involution.
\label{idem_dif_2}
\end{proposition}

Now that we know the general shape of continuous idempotent functions, we would like to have a classification up to topological conjugacy. Unfortunately, we will see that there are many equivalence classes and such fact complicates having a simple classification.

A classical result of cardinal theory states that the set of continuous real value functions has the same cardinal as the real numbers, i. e. $|\{f:\mathbb R \rightarrow \mathbb R\hspace{1mm} : \hspace{1mm} f  \textnormal{ continuous}\}|= |\mathbb R|$ (see \cite[Chapter 1.5, Exercise 23]{cardinal_arithmetic}).
Thus, the set of idempotent continuous functions has at most the continuum cardinality. We will see that in fact, there are exactly $|\mathbb R|$ equivalence classes of idempotent continuous functions up to topological conjugacy.

 \begin{ex}
There exists a family of idempotent continuous functions not topologically conjugated with each other with cardinality $|\mathbb R|$.
 \label{idemp_exemple}
 \end{ex}
 \begin{proof}
Let $\{f_\lambda\}_{\lambda \in (0,1)}$ be the family of functions defined as follows. In $(-\infty,1]$ we have
\begin{equation}
f_\lambda (x) =
\left\{
	\begin{array}{ll}
		0  & \mbox{if } x \leq 0, \\
		x & \mbox{if } 0 < x \leq 1.
	\end{array}
\right.
\end{equation}
For all $m\in \mathbb N\setminus\{0\}$, $f_\lambda (m+1)$ takes the value of the $m$th decimal position of $\lambda$'s binary representation (unique with the convention that it does not end with 1 repeating). Finally, we define the value of $f_\lambda$ in $(1, \infty)\setminus \mathbb N$ in such a way that $f$ is continuous and $f((1,\infty)\setminus \mathbb N)\subset (0,1)$. For instance, $f$ can be a line in $(m,m+1)$ when $f_\lambda (m)\not = f_\lambda (m+1)$ and a parabola when $f_\lambda (m) = f_\lambda (m+1)$, as depicted in Figure \ref{Figura_grafica_f_lambda}. It is clear that all these functions are idempotent and continuous.

\begin{figure}[htbp]
    \centering
        {\includegraphics[width=0.835\textwidth]{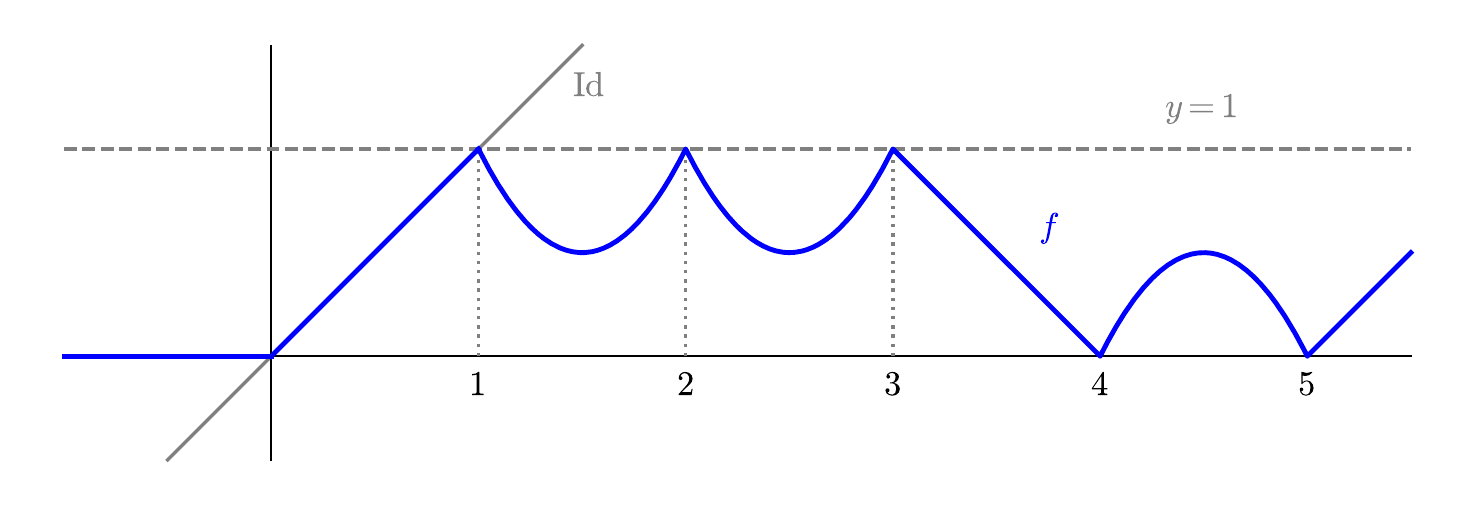}}
    \caption{Graph of the function $f_\lambda$ with $\lambda=(0.11001\dots)_2$.}
    \label{Figura_grafica_f_lambda}
\end{figure}

Let $\lambda, \mu \in (0,1)$ and  $\varphi$ a homeomorphism such that $\varphi \circ f_\lambda=f_\mu\circ \varphi$, we will show that $\lambda =\mu$.

Since $\varphi$ is a homeomorphism it sends fixed points of $f_\lambda$ to fixed points of $f_\mu$. Hence, $ \varphi([0,1])= [0,1]$ and in particular $\varphi(0) =0$ and $\varphi(1)=1$ or  $\varphi(0)=1$ and $\varphi(1)=0$.

If $\varphi(0)=1$ and $\varphi(1)=0$ then $\varphi$ is decreasing and we have
\[1=\varphi(0)=\varphi(f_\lambda ((-\infty,0]))= f_\mu ( \varphi((-\infty,0]))=f_\mu ([1,\infty))=[0,1],\]
which is a contradiction. Hence, $\varphi$ is an increasing homeomorphism with $\varphi(0)=0$ and  $\varphi(1)=1$; in particular, $\varphi ([1,\infty))=[1,\infty)$.
Moreover, we have $f_\mu (\varphi(m))=\varphi(f_\lambda (m))=f_\lambda(m)\in \{0,1\}$, so $f_\mu (\varphi(m))\in \{0,1\}$ for all $m\in \mathbb N $.
By definition of $f_\mu$, $[1,\infty)\cap f_{\mu}^{-1}(\{0,1\})\subset \mathbb N$, and thus $\varphi(m)\in \mathbb N$ for all $m$.

If we show that $\varphi(m) = m $ for all $m\in \mathbb N$, it will follow that $f_\lambda (m)=f_\mu(m)$ and $\lambda = \mu$.

We have already seen that $\varphi(0)=0$ and $\varphi(1)=1$, we prove the general case by induction. Assume that $\varphi(k)=k$ for all $k\leq m-1$, then $\varphi(m)\in \{m,m+1,\dots\} $ since $\varphi(m)\in \mathbb N$. Suppose for the sake of contradiction that $\varphi(m)>m$. Then, by continuity there is $x\in (m-1,m)$ such that $\varphi(x)=m$. Hence,
\[f_\lambda (x)=\varphi^{-1} (f_\mu (\varphi(x)))=f_\mu(m)\in \{0,1\},\]
since $\varphi$ fixes 0 and 1. But $x\not \in f_\lambda^{-1}(\{0,1\})= (-\infty,0]\cup \mathbb N$  which is a contradiction, and thus  $\varphi(m)=m$.
\end{proof}
Note that the functions defined above are solutions of $f^n =f^k$ for any $n,k$ with $n> k\geq 1$ (see Corollary \ref{f^{l_1(n-k)+k}}). Moreover, they are not linearizable, since linear solutions of \eqref{ f^n= f^k}  in $\mathbb R$ are $0$, $-\Ide$ and $\Ide$,  which  have either one or all fixed points.

\subsection{General case}
\label{sec_cas gen_dim 1}

Let $f:I \rightarrow I$ be a continuous solution of \eqref{ f^n= f^k}. By Remark \ref{obs_retract}, $\Ima f^k=J$ is an  interval closed in $I$ and by Proposition \ref{resultat_conjunts}, we have $f^{n-k}_{|J}=\Ide$. Hence, $f_{|J}$ is a periodic function and by Proposition \ref{kerejarto_1dim_I} it is the identity or an involution. Using the characterization given by Proposition \ref{resultat_conjunts}, we obtain the result below.

\begin{proposition}
Let $n,k\in \mathbb N\cup \{ 0\}$ with $n>k$ and $f:I\rightarrow I$ be continuous. If $n-k$ is odd, $f^n=f^k$ if and only if $f^{k+1}=f^k$. If $n-k$ is even, $f^n=f^k$ if and only if $f^{k+2}=f^k$.
\label{f^k+2=f^k or f^k+1=f^k 1dim}
\end{proposition}

Now, Theorem \ref{theorem_A} follows from Proposition \ref{f^k+2=f^k or f^k+1=f^k 1dim}, \ref{idem_dif}, \ref{idem_dif_2} and \ref{kerejarto_1dim_I_invol}.

One can usually check visually if a ``well behaved'' continuous function  is a solution of $f^{k+1}=f^k$. First, $f$ must be the identity in an interval $J\subset I$ closed as a subset. Then we only need to verify that $\Ima f^k=J$. To do so, we can compute successively $f(I), f^2(I), \dots , f^k(I)$ by searching for the maximum and minimum of $f$ in the intervals $I,f(I),\dots, f^{k-1}(I)$. An analogous process can be done for $f^{k+2}=f^k$ with the only difference that $f$ can be an involution instead of the identity in $J$.

Despite this apparently simple structure, Example \ref{idemp_exemple} shows that one should not expect an easy classification up to topological conjugacy of continuous solutions of \eqref{ f^n= f^k} for any $n> k\geq 1$.  We might ask then if imposing more regularity to $f$ is enough to have a nice classification.  If $k=1$, Theorem \ref{theorem_A} answers affirmatively. If $k \geq 2$, the following example shows that there are  many smooth solutions of \eqref{ f^n= f^k}, which difficults such classification.

\begin{ex}
 There exists a family of smooth functions satisfying $f^3=f^2$ not topologically conjugated with each other with cardinality $|\mathbb R|$.
 \label{idemp_exem_smooth}
\end{ex}
\begin{proof}
Let $\{f_\lambda\}_{\lambda \in (0,1)}$ be the family of functions defined as follows. For all $ x\in[-1,0]$, $f_{\lambda}(x)=0$. In $(-\infty,-1]$, $f_\lambda$ is a strictly increasing smooth function that connects smoothly at $-1$ and when  $x\rightarrow -\infty$, $f_{\lambda}(x)\rightarrow -1$. Now, for all $m\in \mathbb N\setminus\{0\}$, $f_\lambda (m)$ takes minus the value of the $m$th decimal position of $\lambda$'s binary representation (unique with the convention that it does not end with 1 repeating). We define $f_\lambda$ everywhere else with smooth transition functions (see  \cite[Page 33]{spivak_smooth}) in such a way that $f_\lambda \in \mathcal C^\infty$ and  $f((1,\infty)\setminus \mathbb N)\subset (-1,0)$, see Figure  \ref{Figura_grafica_f_lambda_smooth}. A simple computation yields $\Ima f= [-1,0]$, $\Ima f^2=\{0\}$ and $f(0)=0$, hence $f^3=f^2$.

\begin{figure}[htbp]
    \centering
        {\includegraphics[width=0.875\textwidth]{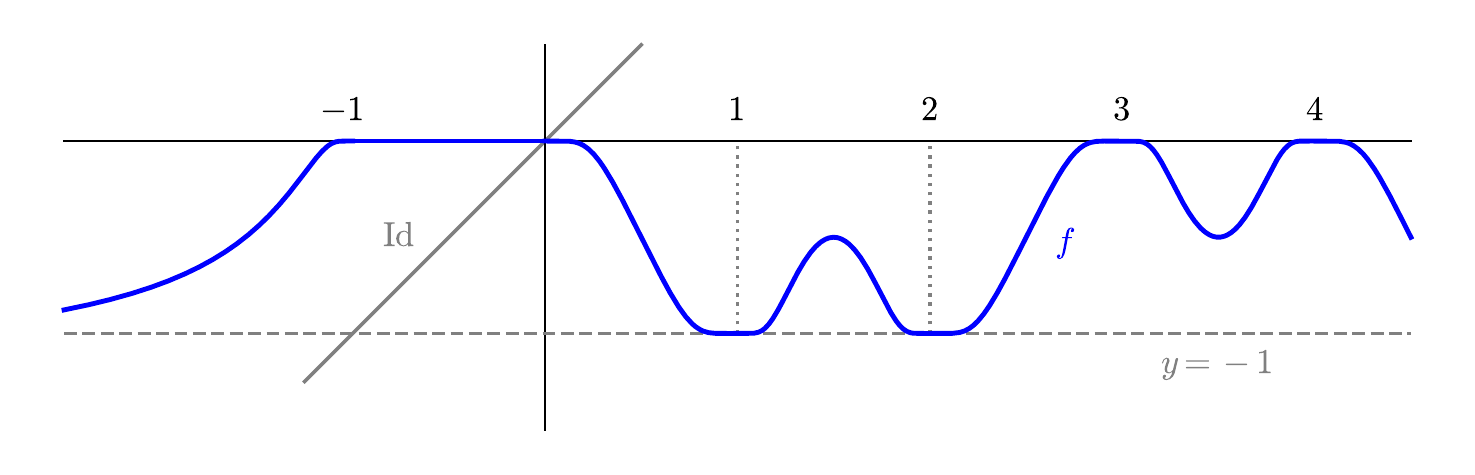}}
    \caption{Graph of the function $f_\lambda$ with $\lambda=(0.11001\dots)_2$.}
    \label{Figura_grafica_f_lambda_smooth}
\end{figure}

Now, if $\varphi$ conjugates $f_\lambda$ with $f_\mu$, then we must have $\varphi([-1,0])=[-1,0]$ as it is the only maximal proper interval whose image under $f_\lambda$ or $f_\mu$ is a point. Then, either  $\varphi(-1)=-1$ and $\varphi(0)=0$ or $\varphi(-1)=0$ and $\varphi(0)=-1$. In the latter case, $\varphi$ is decreasing so $\varphi ([0,\infty))=(-\infty,-1]$, and thus,
\[\varphi([-1,0])=\varphi(f_\lambda ([0,\infty)))= f_\mu ( \varphi([0,\infty)))=f_\mu ((-\infty,-1])=(-1,0],\]
which is a contradiction as the left-hand side is compact and the right-hand side is not.  Hence, $\varphi(-1)=-1$ and $\varphi (0)=0$. Now, exactly the same argument as in Example \ref{idemp_exemple} proves that $\lambda =\mu$.
\end{proof}

As in Example \ref{idemp_exemple}, the functions defined above are solutions of $f^n =f^k$ for any $n,k$ with $n> k\geq 2$ and are not linearizable.

Now we want to tackle the analytic case. We will need the following technical results to do so.

\begin{notation}
Given $a,b\in \mathbb R\cup \{\pm\infty\}$ we will denote by $|a,b|$ any type of interval between $a$ and $b$ which is a subset of the real numbers.

Given an interval $|a,b|$ we will say that $ L$ is a lateral neighborhood of $|a,b|$ if $|a,b|\subset L$ and $L=|c,d)$ or $L=(c,d|$.

Given $A,B\subset \mathbb R$, $a\in \mathbb R$ and $f:A\rightarrow A$, we will say that $f_{|B}\leq a$ if and only if $\forall x\in A\cap B$, $f(x)\leq a$.
\end{notation}

\begin{lemma}
Let $f:|a,b|\rightarrow |a,b|$ be a continuous function such that $f^{k+1}=f^k$ for some $k\in \mathbb N\cup \{ 0\}$ and denote $|c,d|=\Ima f^k$.
Then, there exists $\epsilon>0$ such that $f_{|[d,d+\epsilon)}\leq d$ and $f_{|(c-\epsilon,c]}\geq c$.
\label{no_surt_per_dalt}
\end{lemma}
\begin{proof}
If $b=d$, then clearly the statement concerning $d$ is true. Otherwise we prove it by contradiction. Assume that there does not exist a constant $\epsilon_0>0$ such that $f_{|[d,d+\epsilon_0)}\leq d$. Then, for all $\epsilon_0>0$ there exists  $\epsilon_1>0$ such that $[d,d+\epsilon_1)\subset f( [d,d+\epsilon_0)) $. Applying this repeatedly, it is clear that there is a constant $\delta >0$ such that $[d,d+\delta )\subset f^k([d,b))$ but $\Ima f^k =|c,d|$ and we get a contradiction. We can proceed in the same manner to see the statement concerning $c$.
\end{proof}

\begin{lemma}
Let $f:|a,b|\rightarrow |a,b|$ be a continuous function satisfying  \eqref{ f^n= f^k} and we denote $|c,d|=\Ima f^k$. Then, there exists a lateral neighborhood $L$ of $\Ima f^k $ such that $f(L)= \Ima f^k$. Moreover, if $a=c$,  $b=d$ or, $f^{k+1}=f^k$ and $c\not =d$, we can choose $L$ to be a neighborhood of $\Ima f^k$.
\label{forma_local_unidimensional}
\end{lemma}
\begin{proof}

We will assume $a\not =c$, $b\not =d$ and $c\not =d$, but these cases can be seen similarly. Assuming this, by Remark \ref{obs_retract} we have  $|c,d|=[c,d]$.

We first consider the case $f^{k+1}=f^k$.
 By Lemma  \ref{no_surt_per_dalt} there exists $\epsilon >0$ such that $f_{|[d,d+\epsilon)}\leq d$ and $f_{|(c-\epsilon,c]}\geq c$. Recalling Proposition \ref{resultat_conjunts}, $f_{|[c,d]}=\Ide$, then by continuity we have $f_{|[d,d+\epsilon)}\geq c$ and  $f_{|(c-\epsilon,c]}\leq d$ if $\epsilon$ is small enough. Thus, $L=(c-\epsilon,d+\epsilon)$ is the desired neighborhood.

If $f^{k+1}\not = f^k$, by Proposition \ref{f^k+2=f^k or f^k+1=f^k 1dim} $f^{k+2}=f^k$ and by Proposition \ref{resultat_conjunts} $f_{|[c,d]}$ is an involution. Thus, $f(c)=d$, $f(d)=c$ and by continuity there exists $\epsilon>0$ such that $f_{|[d,d+\epsilon)}\leq d$ and  $f_{|(c-\epsilon,c]}\geq c$. Assume for the sake of contradiction that there does not exist a lateral neighborhood $L$ such that $f(L)=\Ima f^k$. Then, for all $\epsilon>0$ there exists $\delta>0$ such that  $(c-\delta ,c]\subset f([d,d+\epsilon))$ and $[d,d+\delta)\subset f((c-\epsilon,c])$. Hence, for all $\epsilon>0$ there exists $\delta>0$ such that $[d,d+\delta)\subset f^2([d,d+\epsilon))$, which is in contradiction with Lemma \ref{no_surt_per_dalt} applied to $f^2$.
\end{proof}

\begin{proposition}
Let $f:I\rightarrow I$ be a non-periodic differentiable solution of \eqref{ f^n= f^k}. Then, $f^k$ is constant.
\label{diferenciable_dim1}
\end{proposition}
\begin{proof} First, note that  $\Ima f^k\subsetneq I$ since otherwise $f$ would be surjective and hence periodic by Remark \ref{exaustiva}. Assume  for the sake of contradiction that $\Ima f^k=|c,d|$ with $c\not =d$.
If  $f^{k+1}=f^k$, $f$ has the shape shown in Figure \ref{Figura_grafica_f} in a neighborhood of $\Ima f^k$ by Lemma \ref{forma_local_unidimensional}. Then,  it is easy to see that $f$ cannot be differentiable at both $c$ and $d$. Otherwise by Proposition \ref{f^k+2=f^k or f^k+1=f^k 1dim} we have $f^{k+2}=f^k$ and by Proposition \ref{resultat_conjunts}, $f_{||c,d|}$ is an involution. By Remark \ref{f'<0}, involutions have strictly negative derivative and  using Lemma \ref{forma_local_unidimensional} it is easy to see that $f$ cannot be differentible at both $c$ and $d$.
\end{proof}

We are now prepared to give a classification of the analytic solutions of \eqref{ f^n= f^k}. In the non-periodic case we have by Proposition  \ref{diferenciable_dim1} and Lemma  \ref{forma_local_unidimensional} that $f$ must be constant in a proper interval, and since $f$ is analytic it must be constant.

\begin{corollary}
Let $f:I\rightarrow I$ be an non-periodic analytic function satisfying \eqref{ f^n= f^k}. Then, $f$ is constant.
\label{anlitic_dim1}
\end{corollary}

Theorem \ref{theorem_B} follows from Proposition \ref{kerejarto_1dim_I}, \ref{kerejarto_1dim_I_invol} and the previous corollary.

\subsection{On the circle}
\label{on_circle}

The next proposition shows us that the study done for continuous functions $f:I\rightarrow I$ can be applied to functions $f:\mathbb S^1\rightarrow \mathbb S^1$.

\begin{remark}
Any closed arch of $\mathbb S^1$ (where $\mathbb S^1$ is not considered an arch) is diffeomorphic to $[0,1]$. Thus, to study solutions of \eqref{ f^n= f^k} in a closed arch is equivalent to study the solutions in $[0,1]$.
\label{obs_arch}
\end{remark}

\begin{proposition}
Let $f:\mathbb S^1\rightarrow \mathbb S^1$ be a non-periodic continuous function. Then, $f^n=f^k$ for some $n>k\geq 1$ if and only if $\Ima f=J$ is a closed arch and $f_{|J}^{n-1}=f^{k-1}_{|J}$. Moreover, given a closed arch $J$, any continuous function $f: J\rightarrow J $ solution of $f^n=f^k$ with $n>k\geq 1$ can be extended to $\mathbb S^1$ satisfying the same relation.
\label{non_per_S}
\end{proposition}
\begin{proof}
Since $f$ is continuous, and  $\mathbb S^1$ is connected and compact, $\Ima f$ is connected and compact. Moreover, $f$ cannot be surjective by Remark \ref{exaustiva}. Thus, $\Ima f$ must be a closed arch. Using this fact it is not difficult to verify the if and only if statement.

For the second part we may assume without loss of generality that $J$ is a semicircle delimited by the $x$--axis. Then the desired extension can be defined as $F_{|J}=f$ and $F_{|\mathbb S^1\setminus J}=f\circ R$ where $R$ is the reflection through the $x$--axis.
\end{proof}

By the previous proposition, given $f:\mathbb S^1\rightarrow\mathbb S^1$  a non-periodic continuous solution of \eqref{ f^n= f^k},  there exists an arch $K\supsetneq \Ima f$ and clearly $f_{|K}^n=f_{|K}^k$ with $f$ non-periodic in $K$.  Since $K$ can be viewed as an interval,  we get Proposition \ref{idem_dif}, \ref{idem_dif_2}, \ref{diferenciable_dim1} and Corollary \ref{anlitic_dim1} for non-periodic functions defined in $\mathbb S^1$.

\section{Two-dimensional case: Euclidean plane}

In this section we would like to classify up to conjugacy all continuous solutions of \eqref{ f^n= f^k} defined in $\mathbb R^2$. The periodic case is already solved by Kerékjártó's Theorem (see Theorem \ref{kerejarto_R2}), which states that all solutions are linearizable. However, the non-periodic case is much more complex. To see this, recall from Example \ref{idemp_exemple} that there is a family of univariate idempotent continuous functions $\{f_\lambda\}_{\lambda \in (0,1)}$  such that any pair of distinct elements of it is not topologically conjugated to each other.
Then, the family $\{i_1\circ f_\lambda\circ p_1\}_{\lambda \in (0,1)}$  where $p_1$ is the first component projection and $i_1$ is the first component inclusion, has the same property and it is defined in the plane. Indeed,  given an homeomorphism $\varphi:\mathbb R^2\rightarrow \mathbb R^2 $ such that  $ \varphi\circ (i_1\circ f_\lambda\circ p_1) = ( i_1\circ f_\mu\circ p_1 )\circ \varphi $, one may check that  $\phi =p_1\circ \varphi \circ i_1:\Ima f_\lambda\rightarrow \Ima f_\mu $ is a homeomorphism which conjugates $f_\lambda $ with $f_\mu$. Moreover, it is clear that $\phi$  can be extended to a self-homeomorphism of the real line. So if $f_\lambda$ and $f_\mu$ are not conjugated neither are $i_1\circ f_\lambda\circ p_1$ and $i_1\circ f_\mu\circ p_1$.

\begin{remark}
The previous argument also works in higher dimensions, we only need to consider $p_1:\mathbb R^m\rightarrow \mathbb R$ and $i_1:\mathbb R\rightarrow \mathbb R^m$.
\label{exemple_cont_dim_n}
\end{remark}

Thus, we will need to impose more regularity to $f$ to get our desired classification. However, we should not expect a simple classification for solutions of \eqref{ f^n= f^k} when $k\geq 2$, since by Example \ref{idemp_exem_smooth} and the previous argument we will have uncountably many topologically non-equivalent smooth solutions.

First, we study the set $\Ima f^k$, which by Remark \ref{obs_retract} is a retract of the ambient space. It is well known
that a retract of a contractible space is contractible (see \cite[Page 366, Exercise 6]{Topology_Munkres})  and that a $\mathcal C^l $--retract of a  connected $\mathcal C^l$--manifold is a $\mathcal C^l$--submanifold if $l\in \{1,\dots,\infty\}$ (see \cite[Page 20, Exercise 2]{dif_top}).
Thus, we get the following results.

\begin{proposition}
Let $A$ be a contractible space, and $f:A\rightarrow A$ a continuous solution of \eqref{ f^n= f^k}, then $\Ima f^k$ is contractible.
\label{contractil_Imaf^k}
\end{proposition}

\begin{proposition}
Let $M$ be a connected $\mathcal C^l$--manifold, $l\in \{1,2,\dots, \infty\}$ and $f:M\rightarrow M$ a $\mathcal C^l$--solution of \eqref{ f^n= f^k}. Then, $\Ima f^k$ is  a connected $\mathcal C^l$--submanifold of $M$.
\label{retracte_dif}
\end{proposition}

\begin{remark}
In the above proposition, $\dim( M)=\dim (\Ima f^k)$ if and only if $f$ is periodic. Indeed, if $\dim( M)=\dim (\Ima f^k)$ then every point in $\Ima f^k$ has an open neighborhood $U\subset \Ima f^k$ given by its charts, and thus $\Ima f^k$ is open. By
Remark \ref{obs_retract}, $\Ima f^k$ is closed, hence $\Ima f^k=M$ and by Remark \ref{exaustiva}, $f$ is periodic.

\label{dim(M)=dim (Ima f^k)}
\end{remark}

In Proposition \ref{cond_retract} we will state a characterization for  submanifolds of $\mathbb R^m$ which are retracts. From now on, $\{\ast\}$ will denote a singleton.

\begin{proposition}
Let $M$ be a contractible two dimensional $\mathcal C^l$--manifold for some $l\in \{1,\dots,\infty\}$ and let $f:M\rightarrow M$ be a non-periodic $\mathcal C^l$--solution of \eqref{ f^n= f^k}. Then, $\Ima f^k$ is $\mathcal C^l$--diffeomorphic to $\mathbb R$ or $\mathbb R^0=\{\ast\}$.
\label{a el pla}
\end{proposition}
\begin{proof}

By Proposition \ref{retracte_dif} and Remark \ref{dim(M)=dim (Ima f^k)}, $\Ima f^k $ is a connected manifold  with \linebreak$\dim (\Ima f^k)<2$. If $\dim (\Ima f^k)=0$, $\Ima f^k=\{\ast\}$ and if $\dim (\Ima f^k)=1$ by the classification of one dimensional manifolds (see \cite{dif_Milnor}) $\Ima f ^k$ is $\mathcal C^l$--diffeomorphic to $\mathbb R$ or $\mathbb S^1$. However, by Proposition \ref{contractil_Imaf^k} $\Ima f^k$ is contractible, hence $\Ima f ^k\cong \mathbb R$.
\end{proof}

\begin{remark}
The argument above holds if we drop the dimensional condition on $M$ and impose $\dim (\Ima f^k)<2 $.
\label{dim (Ima f^k)<2 imatge}
\end{remark}

Now $\Ima f^k$ is an $f$--invariant set where $f$ is periodic. Hence, when $\Ima f^k\cong \mathbb R$ we can use Proposition \ref{kerejarto_1dim_I} to prove the following statement.

\begin{proposition}
Let $n,k\in \mathbb N\setminus \{0\}$ with $n>k$, $M$ a contractible $\mathcal C^1$--manifold and $f:M\rightarrow M$ a $\mathcal C^1$ function  with $\dim (\Ima f^k)=1$. If $n-k$ is odd, $f^n=f^k$ if and only if $f^{k+1}=f^k$.  If $n-k$ is even, $f^n=f^k$ if and only if $f^{k+2}=f^k$.
\label{f^k+2=f^k or f^k+1=f^k 2 dim}
\end{proposition}

\begin{remark}
If $\dim (\Ima f^k) = 0$, we have $f^{k+1}=f^k$ by Remark \ref{obs_|Ima f^k|=1}.
\end{remark}

To control how the set  $\Ima f^k\cong \mathbb R$ is embedded in the plane we use the following reformulation of \cite[Lemma 3.6]{kerejarto_smooth}.

\begin{lemma}
Let $C\subset \mathbb R^2$ be a $\mathcal C^l $--submanifold of $\mathbb R^2$ for some $l\in \{1,2,\dots, \infty\}$ which is closed as a subset and $C\cong \mathbb R$. Then, there exists a $\mathcal C^l$--diffeomorphism $\varphi: \mathbb R^2\rightarrow\mathbb R^2$  such that $\varphi(C)=\mathbb R \times \{0\}$.

\label{aplanar_R2}
\end{lemma}

We are now prepared to prove that after a conjugation, all $\mathcal C^1$--solutions
of \eqref{ f^n= f^k} in the plane have $\Ima f^k= \mathbb R^i\times \{0\}^{2-i}$ with $i\in\{0,1,2\}$ and $f_{|\Ima f^k}$ is linear. If $f^k$ is constant the result is trivial, and if $f$ is periodic, it is a restatement of Kerékjártó's Theorem. Otherwise, we have:

\begin{theorem}
Let $f:\mathbb R ^2\rightarrow \mathbb R^2$ be a $\mathcal C^l$--solution of \eqref{ f^n= f^k} with $l\in \{1,2,\dots,\infty\}$. Assume that $f$ is non-periodic and $f^k$ is non-constant. Then, $f$ is $\mathcal C^l$--conjugated to a function $g$ with $\Ima g^k=\mathbb R\times \{0\}$ such that

\begin{itemize}
    \item if $n-k$ is odd, $g(x,0)=(x,0)$ for all $x\in \mathbb R$;
\item if $n-k$ is even, either $g(x,0)=(-x,0)$ for all $x\in \mathbb R$, or $g(x,0)=(x,0)$ for all $x\in \mathbb R$.
\end{itemize}
\label{clas_R2}
\end{theorem}
\begin{proof}

 By Proposition \ref{a el pla}, $\Ima f^k\subset \mathbb R^2$ is a $\mathcal C^l$--submanifold $\mathcal C^l$--diffeomorphic to $ \mathbb R$. Moreover, by Remark \ref{obs_retract} $\Ima f^k$ is closed, and thus by the previous lemma there is a $\mathcal C^l$--diffeomorphism $\varphi:\mathbb R ^2\rightarrow \mathbb R^2$ such that $\varphi(\Ima f^k)=\mathbb R \times \{0\}$. If we define $F= \varphi \circ f\circ \varphi^{-1}$, it is clear that $F^n=F^k$ and $\Ima F^k=\mathbb R \times \{0\}$.
 Hence, by Proposition \ref{resultat_conjunts} $F^{n-k}_{|\mathbb R\times \{0\}}=\Ide$.

Now we denote by $i_1$ (resp. $p_1$) the inclusion (resp. projection) respect
the first variable. Clearly  $\mathfrak f=p_1 \circ F\circ i_1$ is a periodic univariate function. Thus, by Proposition \ref{kerejarto_1dim_I}  and \ref{kerejarto_1dim_I_invol} either $\mathfrak f=\Ide$ (this is always the case if $n-k$ is odd) or $\mathfrak f$ is $\mathcal C^l$--conjugated to $-\Ide$ by $\phi_1$. In the first case we take $g=F$. In the later case
we take $\phi(x,y)=(\phi_1(x),y)$ and $g=\phi^{-1}\circ F\circ \phi$. It is easy to check that $g$ has the desired properties.
\end{proof}

Now we can tackle Theorem \ref{rectifica}.

\begin{proof}[Proof of Theorem \ref{rectifica}]
As a particular case of Proposition \ref{f^k+2=f^k or f^k+1=f^k 2 dim} we get $f^n=f$ if and only if $f^3=f$. We only prove the case $f^2=f$ as the other one can be seen similarly.

If $g$ is defined as in the theorem statement, a simple computation yields $g^2=g$. Consider $f$ as in the theorem statement, then by Theorem \ref{clas_R2} there is a function $g$ smoothly conjugated to $f$ such that  $\Ima g= \mathbb R \times \{0\}$ and  $g_{|\mathbb R \times \{0\}}=\Ide$. Then, clearly $g(x,y)=(g_1(x,y),0)$ with $g_{1|\mathbb R \times \{0\}}=\Ide$ and since $g_1 \in \mathcal C^1$ we have
 \begin{equation*}
 \begin{split}
g_1(x_0,y_0)=&  g_1(x_0,0)+   g_1(x_0,y_0) - g_1(x_0,0)= x_0+ \int_0^1\frac{\partial }{\partial t} g_1(x_0,ty_0)dt\\
=& x_0+y_0 \int_0^1\frac{\partial  g_1 }{\partial y}(x_0,ty_0)dt=x_0+y_0\mathfrak g(x_0,y_0),
 \end{split}
 \end{equation*}
where  $\mathfrak g(x_0,y_0) = \int_0^1\frac{\partial g_1}{\partial y}(x_0,ty_0)dt$. Since $ g_1\in \mathcal C^\infty$, we have $\mathfrak g\in \mathcal C^\infty$ and thus, $g$ is as required.
\end{proof}

We should note that different functions  $\mathfrak g\in \mathcal C^\infty$ can define the same map $g$ up to smooth conjugacy.

\begin{remark}
We only use that $f\in \mathcal C^\infty$ in the last line of the proof. Hence, if $f\in \mathcal C^l$ with $1\leq  l<\infty$ we have a similar result. However, in this case $\mathfrak g\not \in \mathcal C^l$, instead it is a particular type of  $\mathcal C^{l-1}$--function.
\end{remark}

We also note that the same argument can be used to prove that, after a $\mathcal C^\infty$--conjugation, the smooth (non-periodic and with $f^k$ non-constant) solutions of \eqref{ f^n= f^k} with $n-k$ odd are exactly the functions $g(x,y)=(x+y\mathfrak g(x,y),g_2(x,y))$, where $\mathfrak g,g_2\in \mathcal C^\infty$ and  $g_2(\Ima g^{k-1})=\{0\}$. If $n-k$ is even then either $g$ is as above or $g(x,y)=(-x+y\mathfrak g(x,y),g_2(x,y))$ with the same conditions on $\mathfrak g$ and $g_2$.
\subsection{Linearizability}

Despite the somewhat satisfying classification of solutions of $f^n=f$
given by Theorem \ref{rectifica}, they need not to be linearizable.

 \begin{ex}
 There exists an infinite family of two variable idempotent polynomial  functions not topologically conjugated with each other and not linearizable.
 \label{polinomis_no_proj}
 \end{ex}
 \begin{proof}
 For $i\in \mathbb N\setminus\{0\}$ consider the polynomial
 \[p_i(x,y)=x\prod_{j=1}^i\frac{y-j}{j},\]
which can be expressed as $x+yq_i(x,y)$ for some polynomial $q_i(x,y)$. Hence, if  $f_i(x,y)=(p_i(x,y),0)$, by Theorem \ref{rectifica}, $f_i^2=f_i$. Moreover, $\Ima f_i=\mathbb R\times \{0\}$ and thus, $f_i$ collapses the space to a line of fixed points. Notice that the only linear maps with such behavior are projections to a line. Additionally, we have,
\[f_i^{-1}(0,0)=p_i^{-1}(0)=(\{0\}\times \mathbb R)\cup \left (\bigcup_{j=1}^i \mathbb R\times \{j\}\right ). \]
That is, the set $f_i^{-1}(0,0)$ is formed by a vertical line and $i $ horizontal ones, and thus it is not a manifold. In particular $f_i$ is not conjugated to a projection. Notice that the first component of $\nabla p_i (x,y)$ only vanishes inside $p_i^{-1}(0)$. Hence, for all $a\in \mathbb R^\ast$,  $f_i^{-1}(a,0)=p_i^{-1}(a)$ is a smooth manifold and it is clear that $f_i^{-1}(b,a)=\emptyset$ for all $b$.

Let $f_i$ and $f_j$ be conjugated by $\varphi$. Then, $\varphi (0,0)=(0,0)$ since $(0,0)$ is the only point where the preimage is not a manifold. Then, $\varphi$ is a homeomorphism between $f_i^{-1}(0,0)$ and $f_j^{-1}(0,0)$ and it is easy to see that they are homeomorphic if and only if $i=j$.
\end{proof}

Examining the previous example one may think that our problem comes from the fact that $\nabla p(z) =(0,0)$ at some points $z$. That is, by Theorem \ref{clas_R2} we know that if $f^2 =f$ and $f\in \mathcal C^1$ then after a $\mathcal C^1$--conjugation we may assume that $f=(g,0)$ with $g(x,0)=x$ (if $f$ is not constant or the identity). We would like to know if in this situation, the condition $\nabla g(z)\not =(0,0)$ for all $z\in \mathbb R^2$ is enough to deduce that $f$ is conjugated to a projection $P$. This condition is very natural, since it assures us that the preimages of $g$, and thus of $f$,  are manifolds. Moreover, if we want the conjugation $\varphi$ from $f$ to $P$, to be a diffeomorphism, this condition is necessary. Indeed, we would have $f=\varphi ^{-1}\circ P \circ \varphi $ and thus for all $z\in \mathbb R^2$,
\[D_{z}f = (D_{P(\varphi(z))}\varphi)^{-1}\cdot D_{\varphi(z)}P\cdot D_{z}\varphi.\]
Since the right hand side of the equality has range 1, $D_{z}f$ has range 1, i.e. \break$\nabla g(z)\not =(0,0)$.

We now show that the condition $\nabla g(z)\not =(0,0)$ for all $z\in \mathbb R^2$ is not enough.

\begin{ex}
Let  $g(x,y)=x+y{x^2}$, then $\nabla g(x,y)\not =(0,0)$ for all $x,y\in \mathbb R $  and the polynomial function $f(x,y)=(g(x,y),0)$ is idempotent and not linearizable.
 \end{ex}
 \begin{proof}
We have $\nabla g(x,y)=(1+2xy,x^2)$ and clearly it never vanishes. Moreover,
\[f^{-1}(0,0)=g^{-1}(0)= \{x=0\}\cup \{xy=-1 \textnormal{ and } x>0\}\cup \{xy=-1 \textnormal{ and } x<0\}.\]
 Thus, the preimage of $(0,0)$ has 3 connected components and therefore $f$ cannot be linearizable.
 \end{proof}

Until this point, we have only considered linear maps defined in the hole plane. Notice, that if a projection is defined in a subset of $\mathbb R^2$ its preimages can have multiple connected components.
If we  accept these kinds of projections we get the following positive result.

\begin{proposition}
Let $\mathbb R \times \{0\}\subset U\subset \mathbb R^2$ be a open set
such that for all $y\in \mathbb R$, $U\cap \mathbb R \times \{y\}$ is connected. Let  $g: U\rightarrow \mathbb R $, $g\in \mathcal C^l$  with $l\in \{1,\dots, \infty\}$ such that $g(x,0)=x$ and $\frac{\partial g}{\partial x}\not = 0$ in $U$. Then, $f=(g,0)$ is $\mathcal C^l$--conjugated  to the projection $P(x,y)=(x,0)$ defined in an open set $V\supset \mathbb R\times \{0\}$.
\label{proj_obert}
\end{proposition}
\begin{proof}
Consider  $\varphi: U \rightarrow \varphi(U)=V$ with $\varphi (x,y)=(g(x,y),y)$. Clearly, $\varphi\in \mathcal C^l$ and,
\begin{align*}
P\circ \varphi (x,y) =& P(g(x,y),y)=(g(x,y),0),\\
\varphi \circ f (x,y)=& \varphi ( g(x,y),0)=(g(g(x,y),0),0)=(g(x,y),0).
\end{align*}
Thus, we only need to check that $\varphi$ is a diffeomorphism. Clearly, it is a local diffeomorphism, since $|D_{(x,y)}\varphi|= \frac{\partial g}{\partial x}(x,y)\not = 0$.

By definition, $\varphi$ is surjective; we check its injectivity. Let $(x_0,y_0), (x_1, y_1) \in U$ such that $\varphi (x_0,y_0)=\varphi( x_1,  y_1) $. That is $g(x_0,y_0)=g(x_1,y_1)$  and  $y_0=y_1$. We consider the function $\mathfrak g(x) = g(x,y_0) $ with derivative $\mathfrak g'(x)=\frac{\partial g}{\partial x}(x,y_0)\not = 0$. Since $\mathfrak g$ is defined in $U\cap \mathbb R\times \{y_0\}$ which is connected, $\mathfrak g$ is monotone. Hence, $\mathfrak g$ is injective and since $\mathfrak g(x_0)=g(x_0,y_0)=g(x_1,y_1)=g(x_1,y_0)=\mathfrak g(x_1)$, we have $x_0=x_1$.
\end{proof}

With the same arguments, one can get an analogous result when $g(x,0)=-x$. In this case, $f$ is conjugated to the map $-P(x,y)=(-x,0)$.

We can now get a local conjugation of $f$ with a
projection defined in the hole plane.

\begin{corollary}
Let  $f: \mathbb R^2\rightarrow \mathbb R^2 $ be a   $\mathcal C^l$--solution of  $f^n=f$ with $l\in\{1,\dots,\infty\}$, non-periodic and non-constant.
Then, $f$ restricted to a neighborhood of $\Ima f$ is conjugated to a projection or a projection composed with a reflection. That is, there exists a neighborhood $V\supset \Ima f$ and a $\mathcal C^l$--diffeomorphism $\phi: V\rightarrow \mathbb R^2$  such that  $\phi \circ f \circ \phi^{-1}=\mathfrak P$, where $\mathfrak P(x,y)=P(x,y)=(x,0)$ or $\mathfrak P(x,y)=-P(x,y)$.
\label{cor_proj_obert}
\end{corollary}
\begin{proof}
By Theorem \ref{clas_R2} we can assume that $f=(g,0)$ and $g(x,0)=x$ or $g(x,0)=-x$ and $g\in \mathcal C^l$. We will only prove the case $g(x,0)=x$, as the other one is analogous. Notice that  $\frac{\partial g}{\partial x} (x,0)\not =0$ for all $x\in \mathbb R$ and by continuity $\frac{\partial g}{\partial x} $ does not vanish in a neighborhood $U\supset \mathbb R\times \{0\}=\Ima f$. If we choose $U$ with a suitable shape, by the previous proposition we have  $\varphi : U \rightarrow \varphi (U)$  which $\mathcal C^l $--conjugates $f_{|U}$ with $P_{|\varphi (U)}$.

The map $\varphi$ is not the desired diffeomorphism since in general $\varphi(U)\not =\mathbb R^2$. To solve this problem, consider $\mathbb R\times \{0\}\subset W\subset \varphi (U)$ open such that  $\partial W$ is the union of the graphics of $h\in \mathcal C^\infty$ and $-h$. Now define $\psi: W \rightarrow \mathbb R^2$ as
\[\psi(x,y)=\left (x,\tan \left (\frac{\pi y}{2h(x)}\right)\right).\]
We have $|D_{(x,y)}\psi|\not = 0$ for all $x,y\in W$ and clearly $\psi$ is bijective, hence it is a $\mathcal C^\infty$--diffeomorphism. Furthermore, since $\psi$ does not change the first component, it conjugates the projection $P:W\rightarrow W$ with the projection $P: \mathbb R^2\rightarrow \mathbb R^2$. Finally, if we consider $V=\varphi^{-1}(W)\supset \Ima f$ we have  $\phi=\psi \circ \varphi :V  \rightarrow \mathbb R^2$ and it conjugates $f_{|V}$ with $P$ defined in the hole plane.
\end{proof}

As a corollary we get Theorem \ref{local_linearizable}. Indeed, if $f$ is constant clearly is conjugated to the zero map and if $f$ is periodic the conjugation is given by  Kerékjártó's Theorem (see Theorem \ref{kerejarto_R2}). Otherwise we can apply Corollary \ref{cor_proj_obert}.

Theorem \ref{local_linearizable}, is not true in the general case \eqref{ f^n= f^k}, where we would clearly choose a neighborhood of $\Ima f^k$ instead of a neighborhood of $\Ima f$. For instance, consider the smooth function $f(x,y)=(f_1(x),0)$ with
\begin{equation}
    f_1(x)=\begin{cases}-e^{\frac{-1}{x}}&\text{if }x>0,\\ 0&\text{if }x\le0.\end{cases}
    \label{e^1/x}
\end{equation}
It is easy to check that $f^3=f^2$ and that for every ball $B_r((0,0))\supset \Ima f^2=\{(0,0)\}$, $f(B_r((0,0)))=(-e^{\frac{-1}{r}},0]\times \{0\}$. Hence, the image of  $f_{|B_r((0,0))}$ is a manifold with boundary and thus it cannot be linearizable.

\subsection{Holomorphic case}
To end this section we consider functions defined in the complex plane. In this context it is natural to consider holomophic solutions of \eqref{ f^n= f^k}. This condition is very restrictive, and for instance it is well known that all non-constant holomorphic functions are open, see \cite[Chapter 3, Theorem 4.4]{stein2003complex}. Thus, if $f$ is a non-constant holomorphic solution of \eqref{ f^n= f^k}, then $\Ima f^k$ is open and since by  Remark \ref{obs_retract} it is also closed, we have $\Ima f^k=\mathbb C$. Therefore, by Remark \ref{exaustiva}, $f$ is periodic and by Proposition \ref{resultat_conjunts}, $f^{n-k}=\Ide$. Now we need the following elemental result  from \cite[Chapter 3, Exercice 14]{stein2003complex}.

\begin{lemma}
Let $f:\mathbb C\rightarrow \mathbb C$ be holomorphic and injective, then $f(z)=az+b$ for some $a,b\in \mathbb C$ with $a\not =0$.
\end{lemma}

Thus, $f(z)=az+b$ with $a\not =0$. If $a=1$, $f$ is a translation and since $f$ is periodic, $f=\Ide$. Otherwise, $a\not =1$, and  $\varphi(z)=z - \frac{b}{1-a}$ conjugates $f$  with $g(z)=az$. Thus, $a^{n-k}=1$ and $g$ is a rotation centered at the origin of angle $\frac{2\pi d}{n-k}$ for some $d\in \mathbb Z$. Since $\varphi$ is a translation, $f$ is a rotation of the same angle centered in $\frac{b}{1-a}$.  Therefore we have proof the following result.

\begin{proposition}
The holomorphic solutions of \eqref{ f^n= f^k} are  rotations of angle $\theta\in \frac{2\pi}{n-k} \mathbb Z$ and constant maps.
\end{proposition}

\section{Higher dimension Euclidean spaces}

If we want to get a nice classification  of the solutions of  \eqref{ f^n= f^k} for all dimension we will have to limit ourselves to a well behaved set of functions. Indeed, in Section \ref{review_periodic} we have seen that not even the  continuous (resp. smooth) periodic case is treatable in  $\mathbb R^3$ (resp. $\mathbb R^7$).

\subsection{Linear maps}

After a linear conjugation, we can assume that linear maps are in their Jordan canonical form. Studying  the Jordan blocks individually one can show that if $f$ is a linear solution of \eqref{ f^n= f^k} its eigenvalues are $(n-k)$th roots of the unity or  0. Moreover, one can show that if a Jordan block has a non-zero eigenvalue then it diagonalizes in $\mathbb C$. Thus, if we let $N_l$ be the nilpotent matrix of dimension $l$ and $R_\theta$ the rotation of angle $\theta$, i.e.
\[ N_3=\begin{pmatrix}
0&1&0\\
0&0&1\\
0&0&0
\end{pmatrix}, \hspace{1.5cm} R_\theta = \begin{pmatrix}
\cos \theta & -\sin \theta \\
\sin \theta & \cos \theta
\end{pmatrix},\]
then, the Jordan blocks of $f$ are of the form $1$, $-1$ (if $n-k$ is even), $R_\theta$ with $\theta \in  \frac{2\pi}{n-k} \mathbb Z$ and $N_l$ with $l\leq k$.

\subsection{Non-periodic \texorpdfstring{$\mathcal C^1$}{C1}--functions}

Many of the results seen in the previous section can be used in higher dimensions.

\begin{proposition}
Let $f:\mathbb R^m\rightarrow \mathbb R^m$ be a $\mathcal C^l$--solution of \eqref{ f^n= f^k} with $l\in \{1,2,\dots, \infty\}$. Then, $\Ima f^k\subset \mathbb R^m$ is a  $\mathcal C^l$--submanifold and if $d= \dim (\Ima f^k)\leq 2$, it is $\mathcal C^l$--diffeomorphic to $\mathbb R^d$.
\label{Imaf^k_R^m}
\end{proposition}
\begin{proof}
By Proposition \ref{retracte_dif} we know that $\Ima f^k$ is a  $\mathcal C^l$--submanifold. Now if \break${\dim(\Ima f^k)<2}$ by Remark  \ref{dim (Ima f^k)<2 imatge} we get the desired result. Assume that $\dim(\Ima f^k)=2$. Then, by Proposition \ref{contractil_Imaf^k} $\Ima f^k$ is contractible and by the classification of 2 dimensional manifolds, it is homomorphic to $\mathbb R^2$ (see \cite{clas_surfaces}). Since in low dimension every topological manifold has a unique $\mathcal C^l$ structure, we have the desired diffeomorphism.
\end{proof}

If $m=3$, as a consequence of the previous proposition and Remark \ref{dim(M)=dim (Ima f^k)}, we get:

\begin{corollary}
Let $f:\mathbb R^3\rightarrow \mathbb R^3$, be a $\mathcal C^l$--solution of  \eqref{ f^n= f^k} with $l\in \{1,2,\dots, \infty\}$. Then, $\Ima f^k$ is a  $\mathcal C^l$--submanifold of $\mathbb R^3$, $\mathcal C^l$--diffeomorphic to $\mathbb R^d$ where  $d=$\break  $ \textnormal{dim} (\Ima f^k )$.
\label{d=0,1,2,3}
\end{corollary}
Let $f$ be as in Proposition \ref{Imaf^k_R^m}. Then by Proposition \ref{resultat_conjunts}, we have $f_{|\Ima f^k} ^{n-k}=\Ide $ and we can use the study of periodic functions in $\mathbb R^d$ with $d\leq 2$ to deduce properties of $f$. If $\Ima f^k =\{\ast\}$, $f^{k+1}=f^k$  by  Remark \ref{obs_|Ima f^k|=1}. We will now study the cases $d=1$ and $d=2$ separately. When doing so, we face knot theory problems, as we will see in the following paragraphs. We emphasize that we will only give ideas on how one may try to deal with them, rather than concrete results.

\subsubsection{One-dimensional  \texorpdfstring{$\Ima f^k$}{Im fk}}
\label{Ima f^k one dim}

By Proposition \ref{f^k+2=f^k or f^k+1=f^k 2 dim} we can limit ourselves
to the study of $f^{k+1}=f^k$ and $f^{k+2}=f^k$ without loss of generality. Now assume that Lemma \ref{aplanar_R2} holds in $\mathbb R^m$. That is, if $l\in \{1,\dots,\infty\}$ assume that  every  $\mathcal C^l$--submanifold of $\mathbb R^m$ which is diffeomorphic to $\mathbb R$ and closed as a subset can be send by an ambient  $\mathcal C^l$--diffeomorphism to the $x$--axis.
Then, if one replaces the conditions ``non-periodic and $f^k$ non-constant'' with  $\dim (\Ima f^k)=1$, it is not hard to see that analogous results to Theorem \ref{clas_R2}, Theorem \ref{rectifica}, Proposition \ref{proj_obert} and Corollary \ref{cor_proj_obert} hold in $\mathbb R^m$. Maybe the only two delicate parts are that in Theorem \ref{rectifica} we would have $g(x_1,\dots,x_m)=(g_1(x_1,\dots,x_m),0,\dots,0)$ with
 \[g_1(x_1,\dots,x_m)=\pm x_1+ x_2\mathfrak g_2(x_1,x_2)+x_3\mathfrak g_3(x_1,x_2,x_3)+\dots +x_m\mathfrak g_m(x_1,\dots,x_m),\]
where $\mathfrak g_i\in \mathcal C^\infty$, and that in Corollary \ref{cor_proj_obert} we consider the cylindrical coordinates $x,r,\theta_1,\dots,\theta_{m-2}$ and
\begin{equation}\partial W= \{(x,r,\theta_1,\dots,\theta_{m-2})\in\mathbb R\times \mathbb R^+\times [0,\pi]^{m-3}\times [0,2\pi)\hspace{1mm}:\hspace{1mm} r=h(x) \}
\label{cylindrical_cord}
.\end{equation}
Sadly, the existence of the diffeomorphism given by Lemma  \ref{aplanar_R2} depends on the ambient dimension, $m$.  To see this we introduce the following non-standard terminology.

\begin{definition}
A \emph{strong embedding} of $\mathbb R^i$ in $\mathbb R^m$ is a smooth submanifold of $\mathbb R^m$ diffeomorphic to $\mathbb R^i$ that is closed as a subset.
\end{definition}
If $m=3$, there are strong embeddings of $\mathbb R$ in $\mathbb R^3$ which cannot be placed in the $x$--axis through an ambient homeomorphism.
The overhand knot with extremes going to infinity is an example, since if we merge the ``endpoints'' we get a trefoil knot. Moreover, the following theorem from \cite{smooth_retracts_euclidian} assures us that any such submanifold can be the image of a smooth retraction.

\begin{proposition}
Let $l\in \{0,\dots,\infty\}$ and $M\subset \mathbb R^m$ be a connected $\mathcal C^l$--submanifold. Then, the following $2s+2$ conditions are equivalent.
\begin{itemize}
    \item  $M$ is a $\mathcal C^r$--retract of some $\mathcal C^r$--retraction for some $r\in \{0,\dots, l\}$.
    \item   $M$ is closed in $\mathbb R^m$ and $\mathcal C^r$--contractible in $M$ for some $r\in \{0,\dots, l\}$.
\end{itemize}
\label{cond_retract}
\end{proposition}

\begin{remark}

We cannot drop the manifold condition on $M$. For instance, the comb space is contractible in itself but it is not a retract of $\mathbb R ^2$.
\end{remark}

We say that two subsets $A,B \subset \mathbb R^m$ are \emph{topologically equivalent} or just \emph{equivalent} if there is an ambient homeomorphism which sends $A$ to $B$. Using standard knots, it is easy to see that there are at least  countably many non-equivalent strong embeddings of $\mathbb R$ in $\mathbb R^3$. Thus, there are at least countable many smooth solutions of $f^2=f$ with $\dim (\Ima f)=1$ in $\mathbb R^3$ not conjugated to each other. We could have seen this through a slight modification of Example \ref{polinomis_no_proj}, nevertheless, the fact that there are a countable number of non-equivalent images is stronger. For instance, it is now clear that a classification such as the one in Theorem \ref{clas_R2} is not possible in $\mathbb R^3$. One could also try to prove that in fact there is an uncountable number of non-equivalent images by smoothly concatenating  the overhand and the figure-eight knot and doing a cantor diagonal argument (see Figure \ref{figure_nus}).

\begin{figure}[htbp]
    \centering
      \includegraphics[width=0.85\textwidth]{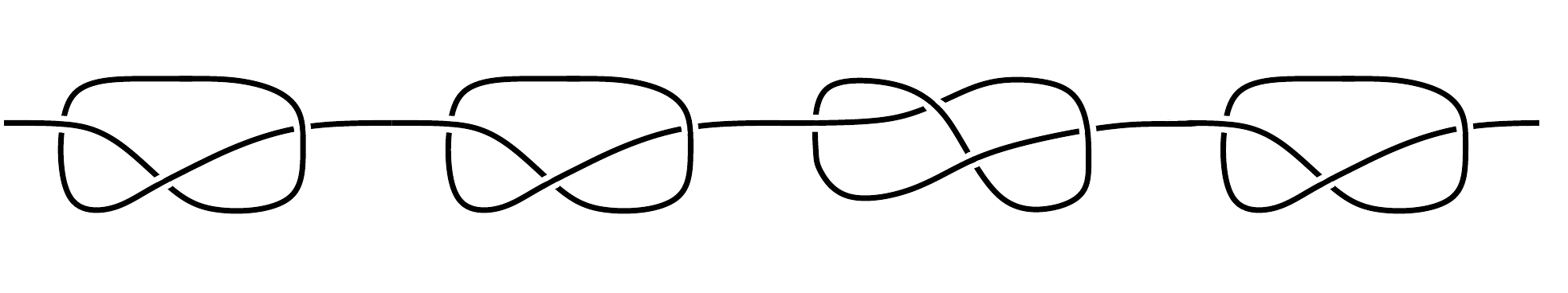}
      \caption{Concatenation of two overhand knots, one figure eight knot and one overhand knot.}
      \label{figure_nus}
 \end{figure}

If $m\geq 4$ it is well known that all smooth knots in $\mathbb S^m$ are smoothly trivial, i.e. there is a smooth diffeomorphism which sends them to an equator (see \cite{unknotted_spheres}).
Now suppose that for some compact subspace $K$,  $\Ima f^k\cap K^c$ is contained in the $x$--axis (after a  $\mathcal C^l$--endomorphism of $\mathbb R^m$). Then if we  compactificate $\mathbb R^m$ to $\mathbb S^m$ adding the point at infinitum, $\Ima f^k\cup \{\infty\}$ is sent to an smooth embedding of $\mathbb S^1$ in $\mathbb S^m$, i.e. a smooth knot in $\mathbb S^m$. The diffeomorphism that sends this knot to the equator can be used to create a diffeomorphism in $\mathbb R^m$ which sends $\Ima f^k$ to $\mathbb R\times \{0\}^{m-1}$.

Notice that for knots such as the one shown in Figure \ref{figure_nus}, it is not clear that outside a compact set we can diffeomorphically send them to the $x$--axis. Hence, the approach described in the previous paragraph does not work in general. However, we believe that with a direct approach one can prove that any strong embedding of $\mathbb R$ in $\mathbb R^m$ can be sent by a global diffeomorphism to $\mathbb R\times \{0\}^{m-1}$ when $m\geq4$.

\subsubsection{Two-dimensional \texorpdfstring{$\Ima f^k$}{Im fk}}
\label{sec_bid_Imaf}

By Proposition \ref{Imaf^k_R^m}, $\Ima f^k\cong \mathbb R^2$.  Assume now that $\Ima f^k=\mathbb R^2\times \{0\}^{m-2}$.
Then, following the proof of Theorem \ref{clas_R2} -- while using Theorem \ref{kerejarto_R2} instead of Proposition \ref{kerejarto_1dim_I_invol} -- one shows that after a $\mathcal C^l$--conjugation  $f(x,\bm 0)= (Mx,\bm 0)$, where $M$ is a  periodic linear map, $x\in \mathbb R^2$ and $\bm{0}=(0,\dots,0)\in \mathbb R^{m-2}$ (we do not distinguish between $\mathbb R^m$ and $\mathbb R^2\times \mathbb R^{m-2}$). Similarly, if one also has $f^n=f$ and $f\in \mathcal C^\infty$ it can be shown with arguments from Theorem \ref{rectifica}  that after a smooth conjugation,
\[f(x_1,x_2,\dots,x_m)=\left (M\begin{pmatrix}x_1\\x_2\end{pmatrix}+ x_3\mathfrak g_3(x_1,x_2,x_3)+\dots +x_m\mathfrak g_3(x_1,\dots,x_m),\bm0\right ),\]
where   $\mathfrak g_i:\mathbb R^m\rightarrow\mathbb R^2$ are smooth functions. We also have analogous results in the local context:

\begin{proposition}
Let  $f:\mathbb R^m\rightarrow \mathbb R^m  $ be a  $\mathcal C^l$--solution of  $f^n=f$ with $l\in\{1,\dots,\infty\}$ and $\Ima f = \mathbb R ^2\times \{0\}^{m-2}$.
 Then, $f$ restricted to a neighborhood of $\Ima f$ is linearizable. That is, there exists a neighborhood $V\supset \Ima f$ and a $\mathcal C^l$--diffeomorphism $\phi: V\rightarrow \mathbb R^m$  such that  $\phi \circ f \circ \phi^{-1}= L$, where $L(x,y)=(Mx,\bm0)$ with $M$ a  periodic linear map, $x\in \mathbb R^2$ and $y\in \mathbb R^{m-2}$.
\end{proposition}
\begin{proof}

Let $x\in \mathbb R^2$ and $y\in \mathbb R^{m-2}$, then it is clear that $f(x,y)=(g(x,y),\bm0)$ for a certain function $g=(g_1,g_2):\mathbb R^m\rightarrow \mathbb R^2$. By the arguments made in the paragraph above we can assume without loss of generality that $g(x,\bm0)=Mx$. Now, given $z\in \mathbb R^m$, we define
\[v(z)=\left (\frac{\partial g_1}{\partial x_1}(z),\frac{\partial g_1}{\partial x_2}(z)\right ),\hspace{1cm} w(z)=\left (\frac{\partial g_2}{\partial x_1}(z),\frac{\partial g_2}{\partial x_2}(z)\right ).\]
Then, since $\begin{psmallmatrix}v(x,\bm0)\\w(x,\bm0) \end{psmallmatrix}=M$ for all $x\in \mathbb R^2$, there is a convex neighborhood $U$ of $\mathbb R ^2\times \{0\}^{m-2}$ such that for all $z,t\in U$, $ \textnormal{det}(v(z)^T,w(t)^T)\not =0$. That is, in $U$, the range of directions of  $v$ and  of $w$ do not intersect. Now we follow Proposition \ref{proj_obert}. We define $\varphi : U\rightarrow \varphi (U)$ as $\varphi (x,y)=(g(x,y),y)$ and we can check that it conjugates $f$ with $L$, it is surjective and a local $\mathcal C^l$--diffeomorphism. To prove that $\varphi$ is injective we slightly modify the argument made in Proposition \ref{proj_obert}. If $\varphi(x_0,y_0)=\varphi(x_1,y_1)$ it is clear that $y_0=y_1$. Then, if $x_0\not =x_1$, by the definition of $U$, either $x_0-x_1\not \in \langle\{v(z)^\perp\hspace{1mm}:\hspace{1mm} z\in U \}\rangle$ or  $x_0-x_1\not \in \langle\{w(z)^\perp\hspace{1mm}:\hspace{1mm} z\in U \}\rangle$. Assume without loss of generality that we are in the former case and notice that $v(x,y)$ represents the gradient at $x$ of $g_{1|\mathbb R^2\times \{y\}}:\mathbb R^2\rightarrow \mathbb R$. Thus, in  $[x_0,x_1]\subset \mathbb R^2$ (the segment that joins $x_0$ and $x_1$)  $g_{1|\mathbb R^2\times \{y_0\}}$  is monotonous and we get a contradiction since $g(x_0,y_0)=g(x_1,y_1)=g(x_1,y_0)$.

Finally, notice that $\varphi$ is not the desired map since   $\varphi(U)\not =\mathbb R^m$. To solve this we use the same techniques as in Corollary \ref{cor_proj_obert} while considering a modified version of the cylindrical coordinates from equation \eqref{cylindrical_cord}.
\end{proof}

Now, we would like to prove that there exists a global diffeomorphism which sends $\Ima f^k$ to $\mathbb R^2\times \{0\}^{m-2}$, in this case we say that $\Ima f^k$ is \emph{trivially embedded}.

If $m=4$, since there are knotted spheres in $\mathbb S^4$, one may show (as we did for $\mathbb S^1$ embedded in $\mathbb S^3$ with the overhand knot) that there are strong  embeddings of $\mathbb R^2$ in $\mathbb R^4$ not trivially embedded.

If $m\geq 5$ then  in \cite{unknotted_spheres} it is shown that any smooth embedding of $\mathbb S^2$ in $\mathbb S^m$ is smoothly trivial. Again we believe (especially if $m\geq 6$) that the same is true for strong  embeddings of $\mathbb R^2$ in $\mathbb R^m$.

If $m=3$, we believe that the following construction is a strong embedding of $\mathbb R^2$ in $\mathbb R^3$ which is not trivially embedded. Consider the embedding of $[0,\infty)$ into $\mathbb R^3$ formed by the smooth concatenation of overhand knots. Thicken this up to an embedding of $[0,\infty) \times \mathbb D^2$ into $\mathbb R^3$. The boundary of this construction is a strong embedding of $\mathbb R^2$ in $\mathbb R^3$ (after smoothing the edge $\{0\}\times \mathbb S^1$) and we believe it is not  trivially embedded. Another candidate is given in the first paragraph of \cite{TuckerThomasW.1977OTFS}.

\section{Other manifolds}
\label{manifolds}
We would like to study solutions of \eqref{ f^n= f^k} defined in other manifolds (for $\mathbb S^1$ see Section  \ref{on_circle}). We now state a useful result in this scenario.

\begin{proposition}
Let $M$ be a manifold and $f:M\rightarrow M$ a continuous solution of \eqref{ f^n= f^k}. Then, $\pi_l(\Ima f^k)$ is isomorphic to a subgroup of  $\pi_l(M)$ for all $l\geq 1$.
\label{prop_grup_fonamental}
\end{proposition}
\begin{proof}
Denote $h= f^{k(n-k)}$, by Proposition \ref{h^2= h}, $\Ima h=\Ima f^k$ and $h_{|\Ima f^k}=\Ide$. Then we have the following commutative diagrams,
\[\begin{tikzcd}\Ima f^k\rar{i}\arrow[ bend right]{rr}[black,swap]{\Ide} & M\rar{h} & \Ima f^k
\end{tikzcd} \hspace{0.3cm}\Longrightarrow \hspace{0.3cm}
\begin{tikzcd}\pi_l(\Ima f^k)\rar{i_\ast}\arrow[ bend right]{rr}[black,swap]{\Ide_\ast=\Ide} & \pi_l(M)\rar{h_\ast} & \pi_l(\Ima f^k)
\end{tikzcd}.\]
Where $i_\ast$ and  $h_\ast$ are group morphisms, and since $h_\ast \circ i_\ast=\Ide$, we know that  $i_\ast $ is injective and $h_\ast $ surjective. Then, $i_\ast$ is the desired isomorphism.
\end{proof}

The same arguments work for other homologies and cohomologies.

Looking at the problems faced in Sections \ref{Ima f^k one dim} and \ref{sec_bid_Imaf}, one may think that when studying solutions of \eqref{ f^n= f^k} in $\mathbb S^m$ the unknotting results of $\mathbb S^1$ or $\mathbb S^2$ in $\mathbb S^m$ for $m\geq 5$ would be very useful.  However, the following observation shows that this is not the case.

\begin{remark}
If $f:\mathbb S^m\rightarrow\mathbb S^m$ is a continuous solution of \eqref{ f^n= f^k}, then $\Ima f^k\not \cong \mathbb S^r$ for all $r\not =m$
\label{obs_S^m_grups}
\end{remark}

To prove this notice that for all $s>0$,  $\pi_s(\mathbb S^s)\cong \mathbb Z$, $\pi_l(\mathbb S^s)\cong 0$ if $l<s$, and use Proposition \ref{prop_grup_fonamental}.

Nevertheless, the study of $\mathcal C^1$--solutions of \eqref{ f^n= f^k} in $\mathbb S^2$ and $\mathbb S^3$ is quite simple.

\begin{proposition}
Let $f:\mathbb S^2\rightarrow \mathbb S^2$, be a $\mathcal C^1$--solution of \eqref{ f^n= f^k} where $f^k$ is not constant. Then, $f$ is topologically conjugated to an element of the orthogonal group $O(3)$.
\label{S^2 f^n=f^k}
\end{proposition}
\begin{proof}
By Proposition \ref{retracte_dif}, $\Ima f^k$ is compact connected manifold with $d=$ \break ${\dim (\Ima f^k)\leq 2}$. If $d=0$ clearly $f^k$ is constant. If $d=2$, by Remark \ref{dim(M)=dim (Ima f^k)} $f$ is periodic and by Theorem \ref{kerejarto en S^2} it is topologically conjugated to an element of $O(3)$. Finally, by Remark \ref{obs_S^m_grups} we cannot have $d=1$ since then $\Ima f^k\cong \mathbb S^1$.
\end{proof}

\begin{proposition}
Let $f:\mathbb S^3\rightarrow \mathbb S^3$, be a $\mathcal C^1$--solution of \eqref{ f^n= f^k} where $f^k$ is not constant. Then, $f$ is periodic.
\end{proposition}
\begin{proof}
By Proposition \ref{retracte_dif}, $\Ima f^k$ is a compact connected manifold with $d=$ \break$\dim (\Ima f^k)\leq 3$. When $d=0,1$ the argument in Proposition \ref{S^2 f^n=f^k} hold. When $d=3$, $f$ is periodic by  Remark \ref{dim(M)=dim (Ima f^k)}. Finally, if $d=2$, then $\Ima f^k$ is a compact connected 2--manifold. By their classification it is a sphere, the connected sum of projective planes or the connected sum of torus. All these surfaces have a non trivial fundamental group, but $\pi_1(\mathbb S^3)$ is trivial. Thus,  $d\not =2$ by Proposition \ref{prop_grup_fonamental}.
\end{proof}

Finally, we would like to point out that in a torus we can have $\Ima f^k\cong \mathbb S^1$. For instance, take $f:\mathbb S^1\times \mathbb S^1\rightarrow \mathbb S^1\times \mathbb S^1$ defined as $f(x,y)=(x,N)$ where $N$ is the north pole.

\section{Hardy-Weinberg equilibrium}
\label{section_hardy}

The equilibrium of Hardy-Weinberg in its simplest form states that, under certain biological assumptions (see \cite{Yap_simpler}), the proportions of each allelic pair $AA$, $Aa$ and $aa$ in a population with two alleles $A$ and $a$ is constant through time. This result is frequently used in genetic studies, since it allows us to deduce the proportion of each allelic pair only by knowing the proportion of one of them (see \cite[Chapter 11]{Population_dynamics}). We will show that this equilibrium follows from the idempotent nature of the ``offspring'' function.

 Let $a_1,\dots, a_k$  be the alleles of our population and denote by $a_ia_j$ or  $a_ja_i$ the allelic pair with alleles $a_i$ and $a_j$. Then we have $\binom{k+1}{2}$ different types of allelic pairs. Let $x_{i,j}=x_{j,i}$ represent the proportion of the population with allelic pair $a_ia_j$, thus $\sum_{i\leq j}x_{i,j}=1$ and we may  think of $x_{k,k}$ as an affine combination of the other proportions $x_{i,j}$. Let $x\in E=\mathbb R^{\binom{k+1}{2} -1}$, where
\[x=(x_{1,1},x_{1,2},\dots, x_{1,k}, x_{2,2},x_{2,3},\dots,x_{
k-1,k-1},x_{
k-1,k}).\]
Define now the functions $p_i:E\rightarrow \mathbb R$ for $i<k$ as the proportions of the allele $a_i$ in the whole population, that is
 \[p_i(x)=\frac{1}{2}\sum_{j=1}^kx_{i,j} + \frac{1}{2}x_{i,i}.\]
Let $f:E\rightarrow E$, with $f(x)=(f_{i,j}(x))_{i\leq j, i\not =k}$ represent the proportions of each type of allelic pair in $x$'s  offspring. Then, in \cite{Yap_simpler} it is shown that with certain biological assumptions,
\begin{equation*}
  f_{i,j}(x)= \left\{
     \begin{array}{cc}
       2p_i(x)p_j(x),  & i\not = j,\\
       p_i(x)^2,  & i=j.\\
    \end{array}
   \right.
\end{equation*}
Moreover, either by a biological argument or an algebraic computation, it is easy to see that $p_i(f(x))=p_i(x)$. Since $f(x)$ is defined through the values $p_i(x)$, it is clear that $f(f(x))=f(x)$, i.e. $f$ is an idempotent function. Thus, the proportions of allelic pairs can only change from the first to the second generation, and since in biology we do not observe this, we get the desired equilibrium. Furthermore, one can check that $\varphi: E\rightarrow E$ defined as $\varphi(x)=(\varphi_{i,j}(x))_{i\leq j, i\not =k}$ with,
\[
\varphi_{i,j}(x)= \left\{
     \begin{array}{cc}
       x_{i,j}-2p_i(x)p_j(x),  & i\not = j,\\
       p_i(x),  & i=j,\\
    \end{array}
   \right.
\]
bipolynomically conjugates $f$ to a projection of range $k-1$. As we have seen in Example \ref{idemp_exemple} this is a stronger result than the fact that $f$ is idempotent, i.e. the standard statement of Hardy-Weimberg equilibrium.  It is worth pointing out that one gets the same results for hypothetical species where $l$ progenitors have an offspring with $l$ alleles (one from each parent).

In \cite{Yap_simpler}, biological assumptions are slightly weakened by essentially introducing sexes. That is, divide the population into two groups $M$ and $F$ (each of which has their own allelic proportions $x^M$ and $x^F$) and only pairs of individuals of different groups can have offspring. It is seen that the corresponding ``offspring'' function $f=(f^M,f^F):E\times E\rightarrow E\times E$, is given by,
\[f^M_{i,j}(x)=f^F_{i,j}(x)=  \left\{
     \begin{array}{cc}
       p_i^M(x)p_j^F(x)+p_j^M(x)p_i^F(x),  & i\not = j,\\
       p_i^M(x)p_i^F(x),  & i=j,\\
    \end{array}
   \right.\]
where $x=(x^M,x^F)$ and $p^M_i(x)=p_i(x^M)$ (resp. for $p_i^F$). Moreover, they show that $p^M_i(f(x))=p^F_i(f(x))=\frac{p^M_i(x)+p^F_i(x)}{2}$, hence in $\Ima f$ the conditions to apply the classical Hardy-Weinberg equilibrium are satisfied and they conclude that $f^3=f^2$.

We can show that $f^2\not =f$ by taking $k=2$, $x_{1,1}^M=x_{1,2}^M=\frac{1}{2}$, $x_{1,1}^F=\frac{1}{3}$ and  $x_{1,2}^F=\frac{2}{3}$. The dynamics defined by $f$ can be better understood if we consider the bipolinomial conjugation $\varphi=(\varphi^M,\varphi^F):E\times E\rightarrow E\times E$ given by
\[
\varphi^M_{i,j}(x)= \left\{
     \begin{array}{cc}
       x^M_{i,j}-2a_i(x)a_j(x),  & i\not = j,\\
       a_i(x),  & i=j,\\
    \end{array}
   \right. \hspace{0.7cm}
   \varphi^F_{i,j}(x)= \left\{
     \begin{array}{cc}
       x^M_{i,j}-x^F_{i,j},  & i\not = j,\\
       \frac{p^M_i(x)-p^F_i(x)}{2},  & i=j,\\
    \end{array}
   \right.
\]
where $a_i(x)=\frac{p^M_i(x)+p^F_i(x)}{2}$. Then,
\[(\varphi \circ f\circ \varphi^{-1})_{i,j}^M(y)=\left\{\begin{array}{cc}
       -2y_{i,i}^Fy_{j,j}^F,  & i\not = j,\\
       y_{i,i}^M,  & i=j,\\
    \end{array}\right .
\hspace{0.9cm}
(\varphi \circ f\circ \varphi^{-1})^F(y)=0.\]
Despite the simplistic appearance of the previous equation, $f$ is not linearizable. Indeed, for $k=2$ we have,
\[\begin{split}f^{-1}\left(\left(\frac{3}{16},\frac{5}{8}\right),\left(\frac{3}{16},\frac{5}{8}\right )\right)=&\left\{\left(\left(\lambda,-2\lambda + \frac{1}{2}\right),\left( \mu, -2\mu + \frac{3}{2}\right)\right)\hspace{1mm}:\hspace{1mm} \lambda,\mu \in \mathbb R\right\}\\
&\cup\left\{\left(\left(\lambda, -2\lambda + \frac{3}{2}\right),\left(\mu, -2\mu + \frac{1}{2}\right)\right)\hspace{1mm}:\hspace{1mm} \lambda,\mu \in \mathbb R\right\}.\end{split}\]
 Notice that this set has two connected components and hence $f$ can not be conjugated to a linear map. Moreover, it is easy to check that both connected components intersect the ``biologically relevant'' region. That is, the region where $x_{i,j}^M\geq 0$ for all $i\leq j$ with $i\not =k$ and  $\sum_{i\leq j,i\not =k} x_{i,j}^M\leq  1 $ (resp. for $x_{i,j}^F$). Thus, we cannot conjugate $f$ restricted to this region with a linear map either.

\section*{Acknowledgments}

In concluding this paper I would like to express my hearty thanks to Prof. Armengol Gasull for his patient guidance during the development of this paper. I also thank Mireia Roig Mirapeix for carefully reading the manuscript.

\bibliographystyle{plain}
\bibliography{biblio.bib}

\vspace{10pt}
\emph{E-mail address:} {marc.homsdones@maths.ox.ac.uk}

\end{document}